\documentclass[titlepage]{amsart}
 \usepackage{amsaddr}
\usepackage{amssymb,amsmath,amsthm,latexsym,stmaryrd,cancel}
\usepackage{xcolor}

\usepackage[T1]{fontenc}
\usepackage{lmodern}

\usepackage[pagebackref,colorlinks=true]{hyperref}  
\newtheorem{theorem}{Theorem}[section]
\newtheorem{proposition}{Proposition}[section]
\newtheorem{corollary}{Corollary}[section]
\newtheorem{definition}{Definition}[section]
\newtheorem{remark}{Remark}[section]

\newcommand{\ie}{i.\hspace{.5pt}e.\ }
\newcommand{\f}{\phi}
\newcommand{\g}{\tilde{g}}
\newcommand{\n}{\nabla}
\newcommand{\M}{(\mathcal{M},\A\f,\A\xi,\A\eta,\A{}g)}

\newcommand{\R}{\mathbb R}
\newcommand{\X}{\mathfrak X}

\newcommand{\LL}{\mathcal{L}}

\newcommand{\lm}{\lambda}
\newcommand{\al}{\alpha}
\newcommand{\bt}{\beta}
\newcommand{\A}{\allowbreak{}}
\newcommand{\D}{\mathrm{d}\hspace{-0.5pt}}
\newcommand{\thmref}[1]{Theorem~\ref{#1}}

\newcommand{\propref}[1]{Proposition~\ref{#1}}

\newcommand{\sectref}[1]{Section~\ref{#1}}

\DeclareMathOperator{\tr}{tr} 
\DeclareMathOperator{\Span}{span} 

\begin{document}

\title[Para-Ricci-like Solitons with Vertical Potential on...]
{Para-Ricci-like Solitons with Vertical Potential on Para-Sasaki-Like Riemannian $\Pi$-Manifolds}


\author[H. Manev]{Hristo Manev}

\address[H. Manev]{Medical University of Plovdiv, Faculty of Pharmacy,
Department of Medical Physics and Biophysics, 15A Vasil Aprilov Blvd.,
Plovdiv 4002, Bulgaria}
\email{hristo.manev@mu-plovdiv.bg}

\begin{abstract}
Object of study are para-Ricci-like solitons on para-Sasaki-like almost paracontact almost paracomplex Riemannian manifolds, briefly, Riemannian $\Pi$-manifolds. Different cases when the potential of the soliton is the Reeb vector field or pointwise collinear to it are considered. Some additional geometric properties of the constructed objects are proven. Results for a parallel symmetric second-order covariant tensor on the considered manifolds are obtained. Explicit example of dimension 5 in support of the given assertions is provided.\thanks{H.M. was partially supported by Project MU21-FMI-008 of the Scientific Research Fund, University of Plovdiv, Bulgaria}
\end{abstract}




\maketitle



\section{Introduction}\label{sect-1}

In 1982 R.\,S. Hamilton introduced the concept of Ricci solitons as a special solution of the Ricci flow equation (\cite{Ham82}). In \cite{Cao}, the author made a detailed study on Riemannian Ricci solitons.
The start of the study of Ricci solitons in contact Riemannian geometry is given with \cite{Shar}.
Following this work the investigation of the Ricci solitons in different types of almost contact metric manifolds are done in \cite{GalCra,IngBag,NagPre}.

Different generalizations of this concept are studied: in paracontact geometry \cite{Bla15,PraHad}; in pseudo-Riemannian geometry \cite{BagIng12,BlaPer,Bro-etal,MM-Sol1,MM-Sol2,HM17}.

We investigate the noted concept of Ricci solitons in the geometry of almost paracontact almost paracomplex Riemannian manifolds, briefly, Riemannian $\Pi$-Manifolds. The induced almost product structure on the paracontact distribution of these manifolds is traceless and the restriction on the paracontact distribution of the almost paracontact structure is an almost paracomplex structure. The study of the considered manifolds starts in \cite{ManSta}, where they are called almost paracontact Riemannian manifolds of type $(n,n)$. Their investigation continues in \cite{ManTav57,ManTav2,ManTav3}, under the name almost paracontact almost paracomplex Riemannian manifolds.

In the present paper, we continue the investigation of the introduced in \cite{HM17} generalization of the Ricci soliton called para-Ricci-like soliton. Here, the potential of the considered para-Ricci-like soliton is a vector field, which is pointwise collinear to the Reeb vector field. The paper is organized as follows. After the present introductory \sectref{sect-1}, in \sectref{sect-2} we give some preliminary definitions and facts about para-Sasaki-like Riemannian $\Pi$-manifolds. In \sectref{sect-3} we investigate para-Ricci-like solitons on the considered manifolds and we prove some additional geometric properties. \sectref{sect-4} is devoted to some characterization for para-Ricci-like solitons on para-Sasaki-like Riemannian $\Pi$-manifolds
concerning a parallel symmetric $(0,2)$-tensor. In the final \sectref{sect-5} we comment an explicit example in {support} of some of the {proven} assertions.

\section{Para-Sasaki-like Riemannian $\Pi$-Manifolds}\label{sect-2}

We denote by $\M$ a \emph{Riemannian $\Pi$-manifold}, where $\mathcal{M}$ is a differentiable $(2n+1)$-dimensional manifold, $g$ is a Rie\-mannian metric and $(\f,\xi,\eta)$ is an almost paracontact structure, \ie $\f$ is a (1,1)-tensor field, $\xi$ is a Reeb vector field and $\eta$ is its dual 1-form. The following conditions are valid:
\begin{equation}\label{strM}
\begin{array}{c}
\f\xi = 0,\qquad \f^2 = I - \eta \otimes \xi,\qquad
\eta\circ\f=0,\qquad \eta(\xi)=1,\\ \
\tr \f=0,\qquad g(\f x, \f y) = g(x,y) - \eta(x)\eta(y),
\end{array}
\end{equation}
where $I$ is the identity transformation on $T\mathcal{M}$ (\cite{Sato76,ManTav57}). Consequently, from the latter equalities we obtain the following:
\begin{equation}\label{strM2}
\begin{array}{ll}
g(\f x, y) = g(x,\f y),\qquad &g(x, \xi) = \eta(x),
\\
g(\xi, \xi) = 1,\qquad &\eta(\n_x \xi) = 0,
\end{array}
\end{equation}
where {$\n$ denotes} the Levi--Civita connection of $g$.
Here and further, by $x$, $y$, $z$, $w$ we denote arbitrary vector fields from $\X(\mathcal{M})$ or vectors in $T\mathcal{M}$ at a fixed point of $\mathcal{M}$.

The associated metric $\g$ of $g$ on $\M$ is determined by {the equality}
\begin{equation}\label{gg}
\g(x,y)=g(x,\f y)+\eta(x)\eta(y).
\end{equation}
Obviously, $\g$ is compatible with $\M$ {in the same way} as $g$ and it is indefinite metric of signature $(n + 1, n)$.

%
%
%
%


In \cite{IvMaMa2}, it is introduced and studied {the class} of \emph{para-Sasaki-like spa\-ces} in the set of {Riemannian $\Pi$-manifolds } which are {obtained} from a {specific} cone construction. This special subclass of the considered manifolds is determined by the following condition:
\begin{equation}\label{defSl}
\begin{array}{l}
\left(\nabla_x\f\right)y=-g(x,y)\xi-\eta(y)x+2\eta(x)\eta(y)\xi,\\
\phantom{\left(\nabla_x\f\right)y}=-g(\f x,\f y)\xi-\eta(y)\f^2 x.
\end{array}
\end{equation}


In \cite{IvMaMa2} is proven that the following identities are valid for any para-Sasaki-like Riemannian $\Pi$-manifold:
\begin{equation}\label{curSl}
\begin{array}{ll}
\n_x \xi=\f x, \qquad &\left(\n_x \eta \right)(y)=g(x,\f y),\\
R(x,y)\xi=-\eta(y)x+\eta(x)y, \qquad &R(\xi,y)\xi=\f^2y, \\
\rho(x,\xi)=-2n\, \eta(x),\qquad 				&\rho(\xi,\xi)=-2n,
\end{array}
\end{equation}
where {$R$ and $\rho$ stand} for the curvature tensor and the Ricci tensor, respectively.

It is known from \cite{HM17} that a Riemannian $\Pi$-manifold  $\M$ is said to be
\emph{para-Ein\-stein-like} with constants $(a,b,c)$ if its Ricci tensor $\rho$ satisfies:
\begin{equation}\label{defEl}
\begin{array}{l}
\rho=a\,g +b\,\g +c\,\eta\otimes\eta.
\end{array}
\end{equation}
Moreover, if $b=0$ or $b=c=0$ the manifold is called an \emph{$\eta$-Einstein manifold} or an \emph{Einstein manifold}, respectively. If $a$, $b$, $c$ are functions on $\mathcal{M}$, then the manifold is called \emph{almost para-Einstein-like}, \emph{almost $\eta$-Einstein manifold} or an \emph{almost Einstein manifold}, respectively.

Let is consider a $(2n+1)$-dimensional Riemannian $\Pi$-manifold $\M$ which is para-Sasaki-like and para-Einstein-like with constants $(a,b,c)$. Tracing \eqref{defEl} and using the last equalities of \eqref{curSl}, we have: \cite{HM17}
\begin{equation}\label{tauElSl2}
a+b+c=-2n,\qquad \tau=2n(a-1),
\end{equation}
where $\tau$ stands for the scalar curvature with respect to $g$ of $\M$.
Moreover, for the scalar curvature $\tilde\tau$ with respect to $\g$ on $\M$ we obtain
\begin{equation}\label{abctau*-ElSl}
\tilde\tau=2n(b-1).
\end{equation}

Taking into account \eqref{tauElSl2} and \eqref{abctau*-ElSl}, expression \eqref{defEl} gets the following form
\begin{equation*}\label{defElSl}
\begin{array}{l}
\rho=\left(\dfrac{\tau}{2n}+1\right)g +\left(\dfrac{\tilde\tau}{2n}+1\right)\g
+\left(-2(n+1)-\dfrac{\tau+\tilde\tau}{2n}\right)\eta\otimes \eta.
\end{array}
\end{equation*}

\begin{proposition}\label{prop:El-Dtau}
Let $\M$ be a $(2n+1)$-dimensional para-Sasaki-like Riemannian $\Pi$-manifold. If $\M$ is almost para-Einstein-like  with functions $(a,b,c)$,
then the scalar curvatures $\tau$ and $\tilde\tau$ are constants
\begin{equation*}\label{El-Dtauxi}
\tau = const, \qquad \tilde\tau=-2n
\end{equation*}
and $\M$ is $\eta$-Einstein with constants
\[
(a,b,c)=\left(\frac{\tau}{2n}+1,\,0,\,-2n-1-\frac{\tau}{2n}\right).
\]
\end{proposition}

\begin{proof}
If $\M$ is almost para-Einstein-like then \eqref{tauElSl2} and \eqref{abctau*-ElSl} are valid, where $(a,b,c)$ are a triad of functions.

Using \eqref{curSl} and substituting $y=\xi$, we can express $R(x,\xi)\xi$ as follows
\begin{equation*}\label{Rxxixi-Sl}
R(x,\xi)\xi =	-\eta(x)\xi+\dfrac{1}{2n} Qx
-\dfrac{1}{4n^2}\bigl\{[\tau-2n(2n-1)]\f^2 x - [\tilde\tau+2n]\f x \bigr\}.
\end{equation*}

After that, bearing in mind \eqref{defSl} and \eqref{curSl}, we compute the covariant derivative of $R(x,\xi)\xi$ with respect to $\n_z$ and we take its trace for $z=e_i$ and $x=e_j$ which gives
\begin{equation}\label{trnRxixi-Sl}
\begin{array}{l}
g^{ij}g\bigl(\left(\n_{e_i} R\right)(e_j,\xi)\xi,y\bigr) = -\dfrac{1}{4n}\D\tau(y)
-\left\{\dfrac{\tilde\tau}{2n}+1\right\}\eta(y).
\end{array}
\end{equation}

The following consequence of the second Bianchi identity is valid
\begin{equation}\label{trnRyxixi-Sl}
\begin{array}{l}
g^{ij} g\bigl((\n_{e_i} R)(y,\xi)\xi,e_j\bigr) =
\eta\bigl(\left(\n_y Q\right)\xi\bigr)-\eta\bigl(\left(\n_\xi Q\right)y\bigr).
\end{array}
\end{equation}
For a para-Sasaki-like manifolds, according to \eqref{curSl}, the equalities $Q\xi=-2n\,\xi$ and $\n_x \xi=\f x$ hold. Using them, it follows that $(\n_x Q)\xi=-Q\f x+2n\,\f x$. As a consequence of the latter equality we have that the trace in the left hand side of \eqref{trnRyxixi-Sl} vanishes.
Then, by virtue of \eqref{trnRxixi-Sl} and \eqref{trnRyxixi-Sl} we get
\begin{equation*}\label{Dtauy-Sl}
\D\tau(y) =
-2\{\tilde\tau+2n\}\eta(y),
\end{equation*}
which implies
\[
\D\tau(\xi)=0,\qquad \tilde\tau=-2n.
\]
The latter equalities together with \eqref{tauElSl2} and \eqref{abctau*-ElSl} complete the proof.
\end{proof}

\section{Para-Ricci-like solitons on para-Sasaki-like manifolds}\label{sect-3}
\subsection{Para-Ricci-like solitons with potential Reeb vector field on para-Sasaki-like manifolds}

In \cite{HM17}, it is introduced the notion of the \emph{para-Ricci-like soliton with potential $\xi$}, \ie a Riemannian $\Pi$-manifold $\M$ admits a para-Ricci-like soliton with potential vector field $\xi$ and constants $(\lm,\mu,\nu)$ if its Ricci tensor $\rho$ satisfies the following:
\begin{equation}\label{defRl}
\begin{array}{l}
\rho=-\frac12 \mathcal{L}_{\xi} g - \lm\, g - \mu\, \g - \nu\, \eta\otimes \eta,
\end{array}
\end{equation}
where $\mathcal{L}$ stands for the Lie derivative.
If $\mu=0$ or $\mu=\nu=0$, then \eqref{defRl} defines an \emph{$\eta$-Ricci soliton} or a \emph{Ricci soliton} on $\M$, respectively.
If $\lm$, $\mu$, $\nu$ are functions on $\mathcal{M}$, then the soliton is called \emph{almost para-Ricci-like soliton}, \emph{almost $\eta$-Ricci soliton} or \emph{almost Ricci soliton}.

In \cite{HM17}, it is proved the truthfulness of the following
\begin{theorem}\cite{HM17}\label{thm:RlSl}
Let $\M$ be a $(2n+1)$-dimensional para-Sasaki-like {Riemannian $\Pi$-manifold. } Let $a$, $b$, $c$, $\lm$, $\mu$, $\nu$ be constants {satisfying the following conditions:}
\begin{equation}\label{SlElRl-const}
a+\lm=0,\qquad b+\mu+1=0,\qquad c+\nu-1=0.
\end{equation}
Then, $\M$ admits
a para-Ricci-like soliton with potential $\xi$ and constants $(\lm,\A\mu,\A\nu)$, where $\lm+\mu+\nu=2n$,
if and only if
it is para-Einstein-like with constants $(a,b,c)$, where $a+b+c=-2n$.

In particular, we obtain the following:
\begin{enumerate}
	\item[(i)] $\M$ admits an $\eta$-Ricci soliton with potential $\xi$ and constants $(\lm,0,2n-\lm)$ if and only if $\M$ is para-Einstein-like with constants $(-\lm,-1,\lm-2n+1)$.

	\item[(ii)] $\M$ admits a shrinking Ricci soliton with potential $\xi$ and constants $(2n,0,0)$ if and only if $\M$ is para-Einstein-like with constants $(-2n,-1,1)$.

	\item[(iii)] $\M$ is $\eta$-Einstein with constants $(a,0,-2n-a)$  if and only if
$\M$ admits a para-Ricci-like soliton with potential $\xi$ and constants $(-a,-1,a+2n+1)$.

	\item[(iv)] $\M$ is Einstein with constants $(2n,0,0)$ if and only if
$\M$ admits a para-Ricci-like soliton with potential $\xi$ and constants $(2n,-1,1)$.
\end{enumerate}
\end{theorem}

Now, we study the covariant derivative of the Ricci tensor with respect to the metric $g$ of a $(2n+1)$-dimensional para-Sasaki-like Riemannian $\Pi$-manifold $\M$ with a para-Ricci-like soliton of the considered type.

For a para-Sasaki-like $\M$ we have
\begin{equation}\label{l1}
\left(\mathcal{L}_{\xi} g\right)(x,y)=g(\n_x\xi,y)+g(x,\n_y\xi)=2g(x,\f y).
\end{equation}
Then, bearing in mind the definition equality of $\g$, it follows that
\begin{equation}\label{l2}
\frac12 \mathcal{L}_{\xi} g=\g-\eta\otimes \eta.
\end{equation}
Because of \eqref{defRl}, $\rho$ takes the form
\begin{equation}\label{SlRl-rho}
	\rho = -\lm g - (\mu+1) \g  - (\nu-1) \eta\otimes \eta.
\end{equation}

\begin{corollary}
Let $\M$ satisfy the conditions in the general case of \thmref{thm:RlSl}.
Then, the constants $a$, $b$, $c$, $\lm$, $\mu$, $\nu$ are expressed by $\tau$ and $\tilde\tau$ as follows
\[
\begin{array}{lll}
\lm=-1-\frac{1}{2n}\tau,\qquad &\mu=-2-\frac{1}{2n}\tilde\tau, \qquad &\nu=\frac{1}{2n}(\tau+\tilde\tau)+2n+3,
\\[4pt]
a=\frac{1}{2n}\tau+1,\qquad &b=\frac{1}{2n}\tilde\tau+1, \qquad &c=-2n-2 -\frac{1}{2n}(\tau+\tilde\tau).
\end{array}
\]
\end{corollary}
\begin{proof}
By direct computations from \eqref{SlRl-rho} we complete the proof.
\end{proof}

We apply covariant derivatives to \eqref{SlRl-rho}, using \eqref{gg}, \eqref{defSl} and \eqref{SlElRl-const}, and we get
\begin{equation}\label{SlRl-nrho}
\begin{array}{l}
\left(\n_x \rho\right)(y,z)=(\mu+1) \{ g(\f x, \f y)\eta(z)+g(\f x, \f z)\eta(y)\}\\[4pt]
\phantom{\left(\n_x \rho\right)(y,z)}
-(\mu+\nu)\{ g(x, \f y)\eta(z)+g(x, \f z)\eta(y)\}.
\end{array}
\end{equation}

The Ricci tensor is called $\n$-recurrent if its covariant derivative with respect to $\n$, is expressed only by $\rho$ and some 1-form.

\begin{theorem}\label{thm-1}
Let $\M$ be a $(2n+1)$-dimensional para-Sasaki-like Riemannian $\Pi$-manifold admitting
a para-Ricci-like soliton with potential $\xi$ and constants $(\lm,\mu,\nu)$. Then:
\begin{enumerate}
\item[(i)] Every para-Einstein-like $\M$ is Ricci $\eta$-parallel, \ie $(\n\rho)|_{\ker\eta}=0$.
\item[(ii)] Every para-Einstein-like $\M$ is Ricci parallel along $\xi$ \ie $\n_{\xi}\rho=0$.
\item[(iii)] The manifold $\M$ is locally Ricci symmetric if and only if $(\lm,\mu,\nu)=(2n,-1,1)$, i.e. it is an Einstein manifold.
	\item[(iv)] The Ricci tensor $\rho$ of $\M$ is $\n$-recurrent and satisfies the following formula
\begin{equation}\label{SlRl-nrho2}
\begin{array}{l}
\left(\n_x \rho\right)(y,z)=\dfrac{\lm(\lm-2n)-(\mu+1)^2}{(\mu+1)^2-\lm^2}\{\rho( x, \f y)\eta(z)+ \rho( x, \f z)\eta(y)\}\\[4pt]
\phantom{\left(\n_x \rho\right)(y,z)=}
-\dfrac{2n(\mu+1)}{(\mu+1)^2-\lm^2}\{ \rho(\f x, \f y)\eta(z)+ \rho(\f x, \f z)\eta(y)\},
\end{array}
\end{equation}
where $(\lm,\mu)\neq(0,-1)$.
\end{enumerate}
\end{theorem}
\begin{proof}
The tensors
\(
\left(\n_x \rho\right)(\f y,\f z)\),  
\(\left(\n_{\xi} \rho\right)(y,z)\) and 
\(\left(\n_{x} \rho\right)(\xi,\xi)\) vanish and therefore we finish the proof of \emph{(i)} and \emph{(ii)}.

Bearing in mind \eqref{SlRl-nrho}, the manifold is locally Ricci symmetric, i.e. $\left(\n_x \rho\right)(y,z)=0$, if and only if $1+\mu=\mu+\nu=0$, which is equivalent to $\mu=-\nu=-1$. The value of $\lm=2n$ comes from the condition $\lm+\mu+\nu=2n$ since the manifold is para-Sasaki-like. 
It follows from \thmref{thm:RlSl} (iv) that the manifold is Einstein. So, we prove the assertion \emph{(iii)}.

By virtue of \eqref{strM}, \eqref{strM2}, \eqref{gg} and $\lm+\mu+\nu=2n$, \eqref{SlRl-rho}
can be rewritten as
\[
	\rho(x,y) = -\lm\, g(\f x,\f y) - (\mu+1) g(x,\f y)  -2n\, \eta(x)\eta(y)
\]
and therefore the following two equalities are valid
\[
\begin{array}{l}
	\rho(x,\f y) = -\lm\, g(x,\f y) - (\mu+1) g(\f x,\f y),	
\\[4pt]
	\rho(\f x,\f y) = -\lm\, g(\f x,\f y) - (\mu+1) g(x,\f y).	
\end{array}
\]
The latter two equations for $(\lm,\mu)\neq(0,-1)$ can be solved as a system with respect to $g(\f x,\f y)$ and $g(x,\f y)$ as follows
\[
\begin{array}{l}
	g(x,\f y) = \dfrac{1}{(\mu+1)^2-\lm^2}\{\lm\,\rho(x,\f y) - (\mu+1)\,\rho(\f x,\f y)\},	
\\[4pt]
	g(\f x,\f y) = \dfrac{1}{(\mu+1)^2-\lm^2}\{\lm\rho(\f x,\f y) - (\mu+1) \rho(x,\f y)\}.	
\end{array}
\]
The recurrent dependence \eqref{SlRl-nrho2} of the Ricci tensor is get by substituting the latter equalities into \eqref{SlRl-nrho}. Thus, we complete the proof of \emph{(iv)}.
\end{proof}

\begin{remark}
A para-Sasaki-like Riemannian $\Pi$-manifold $\M$ admitting
a para-Ricci-like soliton with potential $\xi$ and constants $(\lm,\mu,\nu)$ is locally Ricci symmetric just in the case (iv) of \thmref{thm:RlSl}.
\end{remark}

\subsection{Para-Ricci-like solitons with a potential pointwise collinear with the Reeb vector field on para-Sasaki-like manifolds}
Similarly to the definition of a para-Ricci-like soliton with potential $\xi$, given in \eqref{defRl},
we introduce the following notion.
\begin{definition}
{A Riemannian $\Pi$-manifold } $\M$ admits a para-Ricci-like soliton with potential vector field $v$ and constants $(\lm,\mu,\nu)$ if its Ricci tensor $\rho$ satisfies the following:
\begin{equation}\label{defRl-v}
\begin{array}{l}
\rho=-\frac12 \mathcal{L}_{v} g - \lm\, g - \mu\, \g - \nu\, \eta\otimes \eta.
\end{array}
\end{equation}
\end{definition}

Let $\M$ be a para-Sasaki-like Riemannian $\Pi$-manifold admitting a para-Ricci-like soliton whose potential vector field $v$ is pointwise collinear with $\xi$, i.e. $v=k\,\xi$, where $k$ is a differentiable function on $\mathcal{M}$. The vector field $v$ belongs to the vertical distribution $H^\bot=\Span\xi$, which is orthogonal to the contact distribution $H=\ker\eta$ with respect to $g$.

\begin{theorem}\label{thm:k=const}
Let $\M$ be a para-Sasaki-like Riemannian $\Pi$-manifold of dimension $2n+1$ and let it admits a para-Ricci-like soliton with constants $(\lm,\mu,\nu)$ whose potential vector field $v$ satisfies the condition $v=k\,\xi$, i.e. it is pointwise collinear with the Reeb vector field $\xi$, where $k$ is a differentiable function on $\mathcal{M}$. Then:
\begin{enumerate}
	\item[(i)] $k=-\mu$, i.e. $k$ is constant;

	\item[(ii)] $\lm+\nu=k+2n$ is valid;

	\item[(iii)] $\M$ is $\eta$-Einstein with constants $(a,b,c)=(-\lm,0,\lm-2n)$.
\end{enumerate}
\end{theorem}

\begin{proof}
Taking into account the first equality in \eqref{curSl}, in the considered case we have
\[
\begin{array}{l}
\left(\LL_v g\right)(x,y)=g(\n_x v,y)+g(x,\n_y v)=g(\n_x k\xi,y)+g(x,\n_y k\xi)\\[4pt]
\phantom{\left(\LL_v g\right)(x,y)}
=\D{k}(x)\eta(y)+\D{k}(y)\eta(x)+2kg(x,\f y).
\end{array}
\]
Substituting it in \eqref{defRl-v}, we obtain
\begin{equation}\label{SlRl-v}
\begin{array}{l}
\D{k}(x)\eta(y)+\D{k}(y)\eta(x)=-2\{\rho(x,y)+\lm g(x,y)+(k+\mu)g(x,\f y) \\[4pt]
\phantom{\D{k}(x)\eta(y)+\D{k}(y)\eta(x)=-2\{\rho(x,y)}
+(\mu+\nu)\eta(x)\eta(y)\}.
\end{array}
\end{equation}
Using the expression of $\rho(x,\xi)$  from \eqref{curSl} and replacing $y$ with $\xi$, the latter equality implies
\begin{equation}\label{SlRl-v-xk}
\begin{array}{l}
\D{k}(x)= -\{\D{k}(\xi)+2(\lm+\mu+\nu-2n)\}\eta(x).
\end{array}
\end{equation}
Now, substituting $x$ for $\xi$, we get
\begin{equation*}\label{SlRl-v-xik}
\begin{array}{l}
\D{k}(\xi)= -(\lm+\mu+\nu-2n).
\end{array}
\end{equation*}
Therefore, \eqref{SlRl-v-xk} takes the form
\begin{equation}\label{SlRl-v-xk2}
\begin{array}{l}
\D{k}(x)= -(\lm+\mu+\nu-2n)\eta(x).
\end{array}
\end{equation}

Taking into account \eqref{SlRl-v} and \eqref{SlRl-v-xk2}, we obtain the following for the Ricci tensor
\begin{equation}\label{SlRl-v-rho}
\rho=-\lm g -(k+\mu) \g +(\lm+\mu-2n+k)\eta\otimes\eta.
\end{equation}
Therefore, $\M$ is almost para-Einstein-like with functions
\begin{equation}\label{abc1}
(a,b,c)=(-\lm,-k-\mu,\lm+\mu-2n+k).
\end{equation}
%
Then, $\M$ is $\eta$-Einstein with constants
\begin{equation}\label{abc2}
(a,b,c)=\left(\frac{\tau}{2n}+1,\,0,\,-2n-1-\frac{\tau}{2n}\right),
\end{equation}
according to \propref{prop:El-Dtau}.
Comparing \eqref{abc1} and \eqref{abc2}, we deduce that $k=-\mu$, i.e. $k$ is a constant.

Thus, according to \eqref{SlRl-v-xk2}, we infer that the condition $\lm+\mu+\nu=2n$ is satisfied.
Then, \eqref{SlRl-v-rho} takes the following form
\begin{equation}\label{SlRl-v-rho-k=const}
\rho=-\lm g +(\lm+2n)\eta\otimes\eta,
\end{equation}
which completes the proof.
\end{proof}

\subsection{Some additional curvature properties}

Here, we continue to consider a manifold $\M$, $\dim{M}=2n+1$, which is a para-Sasaki-like Riemannian $\Pi$-manifold admitting a para-Ricci-like soliton with vertical potential $v$, i.e. $v=k\,\xi$ for $k=const$. Then, \thmref{thm:k=const} is valid.

Now, we investigate some well-known curvature properties.

A manifold $\M$ is called \emph{locally Ricci symmetry} if $\n\rho$ vanishes. A manifold $M$ is called \emph{Ricci semi-symmetric} if the following equation is valid
\begin{equation}\label{Rrho}
\rho\left(R(x,y)z,w\right)+\rho\left(z,R(x,y)w\right)=0.
\end{equation}

In \cite{Gray78}, the notions of a \emph{cyclic parallel tensor} or a \emph{tensor of Codazzi type} are given, namely the non-vanishing Ricci tensor $\rho$ satisfying the condition
$
(\n_x \rho)(y,z) + (\n_y \rho)(z,x) + (\n_z \rho)(x,y) =0
$
or
$
(\n_x \rho)(y,z) = (\n_y \rho)(x,z)
$, respectively.

In \cite{DeSa08}, it is defined a \emph{Ricci $\f$-symmetric} Ricci operator $Q$, \ie the non-vanishing $Q$ satisfies $\f^2(\n_x Q)y =0$.
Moreover, according to \cite{GhoDe17}, if the latter property is valid for an arbitrary vector field on the manifold or for an orthogonal vector field to $\xi$ the manifold is called \emph{globally Ricci $\f$-symmetric} or \emph{locally Ricci $\f$-symmetric}, respectively.

An \emph{almost pseudo Ricci symmetric manifold} is a manifold whose non-vanishing Ricci tensor has the following condition \cite{ChaKaw07}
\begin{equation}\label{apRs-def}
(\n_x \rho)(y,z) = \{\al(x) + \bt(x)\}\rho(y,z) + \al(y)\rho(x,z) + \al(z)\rho(x,z),
\end{equation}
where $\al$ and $\bt$ are non-vanishing 1-forms.

According to \cite{SinKha01}, a manifold is called \emph{special weakly Ricci symmetric} when its non-vanishing Ricci tensor satisfies the following
\begin{equation}\label{swRs-def}
(\n_x \rho)(y,z) = 2\al(x)\rho(y,z) + \al(y)\rho(x,z) + \al(z)\rho(x,z).
\end{equation}

\begin{theorem}\label{thm-3}
Let $\M$ be a $(2n+1)$-dimensional para-Sasaki-like Riemannian $\Pi$-manifold admitting
a para-Ricci-like soliton with vertical potential $v$ and constants $(\lm,\mu,\nu)$. Then:
\begin{enumerate}
\item[(i)] $\M$ is locally Ricci $\f$-symmetric.
\item[(ii)] each of the following properties of $\M$ is valid if and only if $\M$ is an Einstein manifold: \\
a) locally Ricci symmetric;\quad b) Ricci semi-symmetric;\quad c) globally Ricci $\f$-symmetric;\quad
d) almost pseudo Ricci symmetric;\quad e) special weakly Ricci symmetric;\quad
f) cyclic parallel Ricci tensor;\quad g) Ricci tensor of Codazzi type.
\end{enumerate}
\end{theorem}

\begin{proof}
In a similar way as for \eqref{SlRl-nrho}, taking into account \eqref{SlRl-v-rho-k=const}, we get
\begin{equation}\label{SlRl-nrho-k}
\begin{array}{l}
\left(\n_x \rho\right)(y,z)=(\lm-2n)\{ g(x, \f y)\eta(z)+g(x, \f z)\eta(y)\}.
\end{array}
\end{equation}
Then, it is easy to conclude the statement \emph{(ii-a)}.

Bearing in mind \eqref{SlRl-v-rho-k=const}, it follows from \eqref{Rrho} that
\begin{equation}\label{rhorho}
(\lm-2n)\{R(x,y,z,\xi)\eta(w)+R(x,y,w,\xi)\eta(z)=0.
\end{equation}
Then \eqref{curSl} and \eqref{rhorho} imply
\begin{equation}\label{lgff}
\begin{array}{l}
(\lm-2n)\bigl\{\left[-\eta(x)g(y,z)+\eta(y)g(x,z)\right]\eta(w)\\[4pt]
\phantom{(\lm+2n)\bigl\{}
+\left[-\eta(x)g(y,w)+\eta(y)g(x,w)\right]\eta(z)\bigr\}=0.
\end{array}
\end{equation}
So, \eqref{lgff} for $w=\xi$ provides $\lm=-2n$. Therefore, by virtue of \eqref{Rrho}, $\M$ is Einstein.

The inverse implication is clear which completes the proof of \emph{(ii-b)}.
Similarly to it, we establish the truthfulness of \emph{(i)}, \emph{(ii-c)}, \emph{(ii-f)} and \emph{(ii-g)}.

Substituting \eqref{SlRl-nrho-k} in \eqref{apRs-def} we get
\begin{equation}\label{apRs}
\begin{array}{l}
-\{\al(x)+\bt(x)\}\{\lm g(\f y,\f z)+2n\,\eta(y)\eta(z)\} \\[4pt]
-\al(y)\{\lm g(\f x,\f z)+2n\,\eta(x)\eta(z)\}-\al(z)\{\lm g(\f x,\f y)+2n\,\eta(x)\eta(y)\}\\[4pt]
-(\lm-2n)\{ g(x, \f y)\eta(z)+g(x, \f z)\eta(y)\}=0
\end{array}
\end{equation}
and setting successively $x$, $y$ and $z$ as $\xi$, we obtain that
\begin{equation}\label{albt}
\al=\al(\xi)\eta,\qquad \bt=-3\al(\xi)\eta.
\end{equation}
Setting \eqref{albt} in \eqref{apRs} and substituting $z=\xi$, we get
\[
\lm\al(\xi)g(\f x,\f y)+(\lm-2n) g(x, \f y)=0,
\]
which is fulfilled if and only if $\lm=2n$ and $\al(\xi)=0$.

Vice versa, let $\M$ be Einstein, i.e. $\rho = -2n g$. Then, \eqref{apRs-def} is transformed in
\begin{equation}\label{albtggg}
\{\al(x)+\bt(x)\}g(y,z)+\al(y) g(x, z)+\al(z) g(x,y)=0.
\end{equation}
Substituting successively $x$, $y$ and $z$ for $\xi$,
we get \eqref{albt}, which combined with \eqref{albtggg} implies
\[
\al(\xi)\{-2\eta(x)g(y,z)+\eta(y) g(x, z)+\eta(z) g(x,y)\}=0
\]
for arbitrary $x$, $y$, $z$ and therefore $\al(\xi)=0$ holds.
Thus, we complete the proof of assertion \emph{(ii-d)}.

We come to the conclusion that an almost pseudo Ricci symmetric manifold with $\al=\bt$ is a special weakly Ricci symmetric manifold, comparing \eqref{swRs-def} with \eqref{apRs-def}. Then, from \eqref{albt} we obtain that $\al=0$ and therefore $\M$ has $\n\rho=0$. Taking into account \emph{(ii-d)}, we get the validity of the statement \emph{(ii-e)}.
\end{proof}

\section{Parallel symmetric second order covariant tensor on $\M$}\label{sect-4}

Let $h$ be a symmetric $(0,2)$-tensor field which is parallel with respect to the Levi-Civita connection of $g$, i.e. $\n h=0$. The Ricci identity for $h$ is valid, \ie
\[
\left(\n_x\n_y h\right)(z,w)-\left(\n_y\n_x h\right)(z,w)=-h\left(R(x,y)z,\,w\right)-h\left(z,\,R(x,y)w\right).
\]
The latter equality with $\n h=0$ implies $h\left(R(x,y)z,\,w\right)+h\left(z,\,R(x,y)w\right)=0$.
Therefore the following characteristic of $h$ is valid
\begin{equation}\label{hxi}
	h\left(R(x,y)\xi,\,\xi\right)=0.
\end{equation}

\begin{proposition}\label{prop:h-Sl}
Let $\M$ be a $(2n+1)$-dimensional para-Sasaki-like Riemannian $\Pi$-manifold. Every symmetric second-order covariant tensor which is parallel with respect to the Levi-Civita connection $\nabla$ of the metric $g$
is a constant multiple of this metric.
\end{proposition}

\begin{proof}

Substituting $R(x,y)\xi$ from \eqref{curSl} in \eqref{hxi}, we get $h(x,\xi)\eta(y)-h(y,\xi)\eta(x)=0$. Then, for $y=\xi$ in the latter equality, we have
\begin{equation}\label{hxxi-Sl}
	h(x,\xi)=h(\xi,\xi)\eta(x).
\end{equation}
Bearing in mind the last equality in \eqref{strM2} and \eqref{hxxi-Sl}, we obtain
\begin{equation}\label{hnxxi-Sl}
	h(\n_x \xi,\xi)=0.
\end{equation}
We have
\[
x\bigl(h(\xi,\xi)\bigr)=2h(\n_x \xi,\xi),
\]
because of the expression of $\left(\n_x h\right)(y,z)$ in the case $\n h=0$ for $y=z=\xi$, which means that $h(\xi,\xi)=const$, according to \eqref{hnxxi-Sl}.

Taking the covariant derivative of \eqref{hxxi-Sl} with respect to $y$, we obtain the property
$h(x,\f y)=h(\xi,\xi)g(x,\f y)$, using the first equality of \eqref{curSl}.
Now, substiting $y$ for $\f y$ in the latter equality for $h$ and using \eqref{strM} and \eqref{hxxi-Sl}, we get
\begin{equation}\label{h=g}
h(x,y)=h(\xi,\xi)g(x,y),
\end{equation}
which means that $h$ is a constant multiple of $g$.
\end{proof}

Now, we apply \propref{prop:h-Sl} to a para-Ricci-like soliton.

\begin{theorem}\label{thm:h-Sl}
Let $\M$ be a $(2n+1)$-dimensional para-Sasaki-like Riemannian $\Pi$-manifold and let $h$ be determined as follows
\[
h=\frac12 \mathcal{L}_{\xi} g  + \rho + \mu\, \g  + \nu\, \eta\otimes \eta
\]
for $\mu,\nu\in\R$. The tensor $h$ is parallel with respect to $\n$ of $g$ if and only if $\M$ admits a para-Ricci-like soliton with potential $\xi$ and constants $(\lm,\mu,\nu)$, where
\[
\lm=-h(\xi,\xi)=2n-\mu-\nu.
\]
\end{theorem}

\begin{proof}
By virtue of \eqref{l1} and \eqref{l2}, $h$ takes the form
\begin{equation}\label{h-Sl}
	h=\rho + (\mu+1) \g  + (\nu-1) \eta\otimes \eta.
\end{equation}

Firstly, let $h$ be parallel. Using \eqref{h-Sl}, \eqref{h=g}
and the last equality in \eqref{curSl}, we obtain that
\(
	h=(-2n + \mu + \nu)g.
\) 
Moreover, the latter equality and \eqref{defRl} deduce that it exists a para-Ricci-like soliton with constants $(\lm,\mu,\nu)$, where $\lm=-\mu-\nu-2n$.

Vice versa, the valid condition \eqref{defRl} can be rewritten as $h=-\lm g$.
Taking into account that $\lm$ is constant and $g$ is parallel, it follows that $h$ is also parallel with respect to $\n$ of $g$.
\end{proof}

\section{Example}\label{sect-5}

In \cite{IvMaMa2}, an explicit example of a $5$-dimensional para-Sasaki-like Riemannian $\Pi$-manifold is considered. It is constructed on a Lie group $G$ with a basis of left-invariant vector fields $\{e_0,\dots, e_{4}\}$ with corresponding Lie algebra determined as follows
\begin{equation}\label{comEx1}
\begin{array}{ll}
[e_0,e_1] = p e_2 - e_3 + q e_4,\quad &[e_0,e_2] = - p e_1 - q e_3 - e_4,\\[0pt]
[e_0,e_3] = - e_1  + q e_2 + p e_4,\quad &[e_0,e_4] = - q e_1 - e_2 - p e_3,
\end{array}
\end{equation}
where $p,q\in\R$.
The Lie group $G$ is equipped with an invariant Riemannian $\Pi$-structure $(\phi, \xi, \eta, g)$ {as follows:}
\begin{equation}\label{strEx1}
\begin{array}{l}
g(e_i,e_i)=1, \quad g(e_i,e_j)=0,\quad i,j\in\{0,1,\dots,4\},\; i\neq j, \\[0pt]
\xi=e_0, \quad \f  e_1=e_{3},\quad  \f e_2=e_{4},\quad \f  e_3=e_{1},\quad \f  e_4=e_{2}.
\end{array}
\end{equation}

In \cite{HM17}, it is proven that the considered para-Sasaki-like Riemannian $\Pi$-manifold $(G, \phi, \xi, \eta, g)$ is $\eta$-Einstein with constants
\begin{equation}\label{abcS}
(a,b,c)=(0,0,-4)
\end{equation}
and it admits a para-Ricci-like soliton with potential $\xi$ with constants
\begin{equation}\label{lmnS}
(\lm,\mu,\nu)=(0,-1,5).
\end{equation}

Now, we compute the components of $(\n_i \rho)_{jk}=(\n_{e_i} \rho)(e_j, e_k)$ of $\n\rho$, taking into account \eqref{comEx1}, \eqref{strEx1} and the only non-zero component $\rho_{00}=-4$ of $\rho$. The non-zero of them are determined by the following ones and their symmetry about $j$ and $k$
\begin{equation}\label{nrho-ex}
(\n_1 \rho)_{30}=(\n_2 \rho)_{40}=(\n_3 \rho)_{10}=(\n_4 \rho)_{20}=4.
\end{equation}

In conclusion, the constructed para-Sasaki-like Riemannian $\Pi$-manifold $(G, \phi, \xi, \eta, g)$ with the results in \eqref{abcS}, \eqref{lmnS} and \eqref{nrho-ex} support the proven assertions in \thmref{thm:k=const} for $k=1$, \propref{prop:h-Sl}, \thmref{thm:h-Sl} for $h=0$ and \thmref{thm-1} (i) and (ii).


\end{document}